\documentclass[11pt]{article}
 
\usepackage[T1]{fontenc}
\usepackage{amsmath,amssymb,amsthm,mathtools}
\usepackage[shortlabels]{enumitem} 
 
\usepackage[top=2cm,bottom=2.8cm,left=2.8cm,right=2.8cm,marginparwidth=2.8cm]
{geometry}
 
\usepackage{bm}
\usepackage{aliascnt}
\usepackage[hidelinks]{hyperref}
 
\usepackage[nameinlink,noabbrev]{cleveref}
\newtheorem{theorem}{Theorem}[section]
 
\newaliascnt{lemma}{theorem}
\newtheorem{lemma}[lemma]{Lemma}
\aliascntresetthe{lemma}
\crefname{lemma}{Lemma}{Lemmas}
\Crefname{lemma}{Lemma}{Lemmas}
 
\newaliascnt{proposition}{theorem}
\newtheorem{proposition}[proposition]{Proposition}
\aliascntresetthe{proposition}
\crefname{proposition}{Proposition}{Propositions}
\Crefname{proposition}{Proposition}{Propositions}
 
\newaliascnt{corollary}{theorem}
\newtheorem{corollary}[corollary]{Corollary}
\aliascntresetthe{corollary}
\crefname{corollary}{Corollary}{Corollaries}
\Crefname{corollary}{Corollary}{Corollaries}
 
\theoremstyle{definition}
\newaliascnt{definition}{theorem}
\newtheorem{definition}[definition]{Definition}
\aliascntresetthe{definition}
\crefname{definition}{Definition}{Definitions}
\Crefname{definition}{Definition}{Definitions}
 
\newaliascnt{example}{theorem}
\newtheorem{example}[example]{Example}
\aliascntresetthe{example}
\crefname{example}{Example}{Examples}
\Crefname{example}{Example}{Examples}
 
\newaliascnt{assumption}{theorem}

\aliascntresetthe{assumption}
\crefname{assumption}{Assumption}{Assumptions}
\Crefname{assumption}{Assumption}{Assumptions}
 
\theoremstyle{remark}
\newaliascnt{remark}{theorem}
\newtheorem{remark}[remark]{Remark}
\aliascntresetthe{remark}
\crefname{remark}{Remark}{Remarks}
\Crefname{remark}{Remark}{Remarks}

\newcommand{\1}{\mathbb{I}}
\DeclarePairedDelimiterXPP{\E}[1]{\mathbb E}{\{}{\}}{}{#1}
\newcommand{\Pp}{\mathbb P}
\newcommand{\Law}{\mathcal L}

 \title{Multihomogeneous Measures and Stochastic Polar Representations}

\author{Enkelejd Hashorva\footnote{Department of Actuarial Science,
    University of Lausanne,
    UNIL-Dorigny, 1015 Lausanne, Switzerland,
    Email: Enkelejd.Hashorva@unil.ch}}
   \date{}
 
\begin{document}
\maketitle
 \begin{abstract}
Let \(q\in\mathbb N\), let \(G=(0,\infty)^q\), and let
$ 
S:G\times E\longrightarrow E,
(r,x)\longmapsto S_rx $
be a jointly measurable left action on an arbitrary measurable space
\((E,\mathcal E)\).  For
\(\alpha=(\alpha_1,\ldots,\alpha_q)\in(0,\infty)^q\) set
$ 
\chi_\alpha(r)=\prod_{i=1}^q r_i^{\alpha_i}.
$ 
We study nonzero \(\sigma\)-finite measures \(\nu\) satisfying
$
\nu(S_rA)=\chi_\alpha(r)^{-1}\nu(A),
 r\in G, A\in\mathcal E.$
Motivated by the scalar case \(q=1\) studied in
\cite{EvansMolchanov2018}, we give equivalent conditions for the
existence of an \(E\)-valued random element \(Z\) such that
$
\nu(A)
=
\E*{\int_G\1_A(S_rZ)\prod_{i=1}^q \alpha_i r_i^{-\alpha_i-1}dr_i},
 A\in\mathcal E.
$    We also
characterise when two random elements generate the same homogeneous
measure, using multihomogeneous moments and, after fixing an admissible
product gauge, weighted transverse measures.  When the corresponding
weighted transverse measure is finite, tilting and gauge normalisation
produce a canonical representer, unique   in law on the prescribed gauge shell.  Finally, we characterise stationarity under an
action commuting with \(S\) and construct positive semidefinite
tail-overlap kernels directly from \(\nu\).
\end{abstract}



\section{Introduction}
  \label{1}

  Homogeneous tail measures introduced in \cite{wao,MR3561100} play a
  central role in the theory of regular variation of random processes;
  see, e.g.,
  \cite{DombryHashorvaSoulier2018,Hrovje, PH2020,
  BladtHashorvaShevchenko2022, Guenter,Resnickart, Ilya25,Hashorva2026ShiftGenerated}.  In an abstract setting, these
  measures have been investigated in \cite{EvansMolchanov2018}.  In this
  contribution we continue that line of investigation.  Let therefore
  \(q\ge1\), let \((E,\mathcal E)\) be a measurable space and put
  \[
  G=(0,\infty)^q
  \]
  equipped with its Borel $\sigma$-field \(\mathcal G\).  We let \(G\) act on \(E\)
  through a jointly measurable left action
  \[
  S:G\times E\longrightarrow E,
  \qquad
  (r,x)\longmapsto S_rx.
  \]
  Thus for \(x\in E\),  we have \(S_{\boldsymbol1}x=x, S_r(S_sx)=S_{rs}x\) and   its \(G\)-orbit is $\{S_rx:r\in G\}$. 
  The group operation and all vector products
  and quotients are coordinatewise.  Fix
  \[
  \alpha=(\alpha_1,\ldots,\alpha_q)\in(0,\infty)^q,
  \qquad
  \chi_\alpha(r)=\prod_{i=1}^q r_i^{\alpha_i}
  \]
  and let \(\vartheta_\alpha\) be the product radial measure on \(G\) given by
  \[
  \vartheta_\alpha(dr)
  :=
  \prod_{i=1}^q \alpha_i r_i^{-\alpha_i-1}\,dr_i.
  \]
  A measure \(\nu\) on \((E,\mathcal E)\) is called
  \(\chi_\alpha\)-homogeneous if
  \[
  \nu(S_rA)=\chi_\alpha(r)^{-1}\nu(A),
  \qquad r\in G,\quad A\in\mathcal E.
  \]
  The aforementioned  contributions have dealt only with the case \(q=1\). 

  For an \(E\)-valued random element \(Z\), define the measure \(\nu_Z\)
  on \((E,\mathcal E)\) by
  \begin{equation}
  \label{2}
  \nu_Z(A)
  =
  \E*{\int_G\1_A(S_rZ)\,\vartheta_\alpha(dr)},
  \qquad A\in\mathcal E,
  \end{equation}
  which   is \(\chi_\alpha\)-homogeneous. 
  If for a given \(\nu\) we can find \(Z\) such that \(\nu=\nu_Z\),
  then \(Z\) is called a \emph{stochastic representer} of \(\nu\) and 
  the identity \(\nu=\nu_Z\) is referred to as a
\emph{stochastic polar representation} of \(\nu\), or simply a
\emph{stochastic representation}
  \(\nu\).

  Not every nonzero \(\sigma\)-finite
  \(\chi_\alpha\)-homogeneous measure has a stochastic representer, as
  the following example shows.

  \begin{example}
  \label{3}
  If \(q\ge2\), \(E=(0,\infty)\) and  
  \[
  S_rx=\chi_\alpha(r)x,
  \qquad
  \nu(dx)=x^{-2}\,dx,
  \]
  then \(\nu\) is nonzero, \(\sigma\)-finite and
  \(\chi_\alpha\)-homogeneous.  Defining 
  \[
  t_i=\alpha_i\log r_i,
  \qquad
  u=t_1+\cdots+t_q,
  \qquad
  v=(t_1,\ldots,t_{q-1})
  \] we have 
  \[
  \chi_\alpha(r)=e^u,
  \qquad
  \vartheta_\alpha(dr)=e^{-u}\,du\,dv
  \]
  and consequently for all \(z>0\)
  \[
  \begin{aligned}
  &\int_G\1_{\{1<\chi_\alpha(r)z<2\}}\,
  \vartheta_\alpha(dr)=
  \int_{\mathbb R^{q-1}}dv
  \int_{\log(1/z)}^{\log(2/z)}e^{-u}\,du
  =\infty.
  \end{aligned}
  \]
  Thus any stochastic representer \(Z\) would satisfy
  \[
  \nu_Z((1,2))
  =
  \E*{\int_G
  \1_{\{1<\chi_\alpha(r)Z<2\}}\,\vartheta_\alpha(dr)}
  =
  \infty,
  \]
  whereas \(\nu((1,2))=\int_1^2x^{-2}\,dx=1/2\),  resulting in  a contradiction.
  \end{example}

  For a nonzero \(\chi_\alpha\)-homogeneous measure \(\nu\), write
  \(\nu\in M_{\chi_\alpha}(E)\) if
  \(0<\nu(A_\star)<\infty\) for some \(A_\star\in\mathcal E\), and
  \(\nu\in M_{\chi_\alpha,\sigma}(E)\) when \(\nu\) is also
  \(\sigma\)-finite.\\
  The first natural question is:

  \medskip
  \noindent
  {\bf Q1:} {\it Which measures
  \(\nu\in M_{\chi_\alpha,\sigma}(E)\) admit a stochastic representer
  \(Z\)?}\\ 

  For \(q=1\) polar representations of homogeneous measures are well
  known, see
  \cite{EvansMolchanov2018,DombryHashorvaSoulier2018,
  BladtHashorvaShevchenko2022}.  A set-based polar criterion is given in
  \cite[Prop.~2.8]{EvansMolchanov2018}.  Under a
  measurable-transversal assumption
  \cite[Thm.~2.15]{EvansMolchanov2018} gives another criterion based on
  countably many integrable functions with
  one-sided-continuous orbit maps.  Countable homogeneous coordinates
  and stochastic representers for tail measures are treated in
  \cite[Sec.~3.3]{BladtHashorvaShevchenko2022}, from which we borrow   the scalar forward
  moment identity 
  \cite[Rem.~3.11(ii),
  Eq.~(3.16)]{BladtHashorvaShevchenko2022}.

  Our first result gives four equivalent forms of the existence
  criterion.  They involve the stochastic representation
  \eqref{2}, a strictly positive integrable function with
  finite positive full product-orbit integral \(\nu\)-almost everywhere,
  an admissible product
  gauge and a countable coordinate-equivariant function.  For
  \(q>1\) these conditions exclude the kernel obstruction.  On a
  standard Borel space the orbit-integral condition is precisely the
  total-dissipativity criterion in
  \cite[Thm.~A]{AvrahamReemPeterzil2026}; see the discussion following
  \Cref{10}.  Our theorem gives a direct equivalence with stochastic
  representability, product gauges and detector families on an arbitrary
  measurable space, without invoking the Hopf decomposition.

  Our second question concerns nonuniqueness within the \(\sigma\)-finite
  class.

  \medskip
  \noindent
  {\bf Q2:} {\it If \(Z\) represents
  \(\nu\in M_{\chi_\alpha,\sigma}(E)\), which other random elements
  \(W\) represent $\nu$, and can one select a canonical representer after
  fixing product-orbit coordinates?}\\

  We show that for every \(E\)-valued random element \(W\) 
  \[
  \nu_W=\nu_Z
  \quad\Longleftrightarrow\quad
  \E{\Gamma(W)}=\E{\Gamma(Z)}
  \]
  for every nonnegative \(\chi_\alpha\)-homogeneous measurable
  functional, i.e., for all elements $\Gamma$ of 
  \[
  \mathcal H_\alpha(E)
  :=
  \left\{
  \Gamma:E\to[0,\infty]\ \text{measurable}:
  \Gamma(S_rx)=\chi_\alpha(r)\Gamma(x),\
  r\in G,\ x\in E
  \right\}.
  \]
  We also give a gauge-dependent answer.  Once an admissible product gauge
  is fixed, every representer gives the same weighted transverse measure.
  When this measure is finite, tilting and gauge normalisation produce a
  canonical representer whose law depends only on \(\nu\) and the chosen
  gauge.

  Every stochastically representable measure \(\nu=\nu_Z\) is
  \(s\)-finite, since
  \[
  \nu_Z
  =
  \Phi_\#
  \bigl(\vartheta_\alpha\otimes\Law(Z)\bigr),
  \qquad
  \Phi(r,z)=S_rz
  \]
  is the pushforward of a \(\sigma\)-finite measure.\\ 
    For a measurable
  map \(T:E\to E\) and a measure \(\mu\) on \((E,\mathcal E)\)  we use
  the notation
  \[
  (T_\#\mu)(A):=\mu(T^{-1}A),
  \qquad A\in\mathcal E.
  \]

  The next example shows
  that stochastic representability does not imply \(\sigma\)-finiteness
  and that even within the \(s\)-finite class  equality of represented
  measures need not determine the expectations of homogeneous
  functionals. 

  \begin{example} 
  \label{4}
  Assume \(q\ge2\) and let
  $
  E
  =
  G\mathbin{\dot\cup}
  \bigl((0,\infty)\times\{a,b\}\bigr)
  $ 
  have its disjoint-union $\sigma$-field.  Define
  $
  S_rx=rx,
    x\in G$ 
  and
  \[
  S_r(x,j)=\bigl(\chi_\alpha(r)x,j\bigr),
  \qquad x>0,\quad j\in\{a,b\}.
  \]
  For the measure  
  $
  \beta_\alpha:=(\chi_\alpha)_\#\vartheta_\alpha
  $ 
  direct calculations show that  for any \(B\in\mathcal B((0,\infty))\)
  
  \[
  \beta_\alpha(B)
  =
  \int_{\mathbb R^{q-1}}\int_{\mathbb R}
    \1_B(e^s)e^{-s}\,ds\,dv
  =
  \begin{cases}
  0,
  &\displaystyle\int_B x^{-2}\,dx=0,\\[1mm]
  \infty,
  &\displaystyle\int_B x^{-2}\,dx>0.
  \end{cases}
  \]
  Thus \(\beta_\alpha\) is a nonzero \(0\)-\(\infty\) measure.  It is
  \(s\)-finite but not \(\sigma\)-finite.
  In our notation $\mathcal B(S)$ denotes the Borel $\sigma$-field of a topological space $S$.

  For \(0<t<1/2\) let
  \[
  \Pp\{Z_t=\boldsymbol1\}=\frac12,
  \qquad
  \Pp\{Z_t=(1,a)\}=t,
  \qquad
  \Pp\{Z_t=(1,b)\}=\frac12-t, 
  \]
  where \(\boldsymbol1=(1,\ldots,1)\in G\), whereas
\((1,a)\) and \((1,b)\) belong to the second component
\((0,\infty)\times\{a,b\}\) of \(E\).
   Given   \(A\in\mathcal E\) set
  \[
  A_G=A\cap G,
  \qquad
  A_j=\{x>0:(x,j)\in A\},
  \quad j\in\{a,b\}
  \] 
  and thus 
  \[
  \nu_{Z_t}(A)
  =
  \frac12\vartheta_\alpha(A_G)
  +t\beta_\alpha(A_a)
  +\left(\frac12-t\right)\beta_\alpha(A_b).
  \]
  Since \(c\beta_\alpha=\beta_\alpha\) for every \(c>0\), it follows that
  \[
  \nu_{Z_t}(A)
  =
  \frac12\vartheta_\alpha(A_G)
  +\beta_\alpha(A_a)
  +\beta_\alpha(A_b),
  \]
  which is independent of \(t\).  Hence all \(Z_t\) represent the same
  \(\chi_\alpha\)-homogeneous measure \(\nu\).  This measure is
  \(s\)-finite but not \(\sigma\)-finite, because its restriction to
  either additional component equals \(\beta_\alpha\).  Nevertheless,
  it has sets of finite positive mass; for example 
  \[
  0<
  \nu\bigl((1,2)^q\bigr)
  =
  \frac12\prod_{i=1}^q\bigl(1-2^{-\alpha_i}\bigr)
  <\infty.
  \]
  Moreover we have 
  $ 
  \E{\Gamma_a(Z_t)}=t$ with  
  $$ 
  \Gamma_a(y)
  =
  \begin{cases}
  x, & y=(x,a),\\
  0, & \text{otherwise},
  \end{cases} 
  $$
  which belongs to $\mathcal H_\alpha(E)$ 
  and consequently for \(s\ne t\)
  \[
  \nu_{Z_s}=\nu_{Z_t},
  \qquad
  \E{\Gamma_a(Z_s)}
  \ne
  \E{\Gamma_a(Z_t)},
  \]
  which shows that stochastic representability alone does not imply
  \(\sigma\)-finiteness and that the \(\sigma\)-finiteness assumption in
  \Cref{15} cannot be omitted.
  \end{example}

The results established here have several consequences and
applications.  First, finite-mass threshold sets give rise to positive
semidefinite tail-overlap kernels.  Such kernels arise naturally as
covariance kernels in limit theorems for extremes; for the case \(q=1\),
see, e.g., \cite{kulik:soulier:2020} and the references therein.  This
application is developed in \Cref{32}.\\ 
  Further
applications concern connections with stationary Poisson particle systems that are already studied in \cite{K2010,MolchanovSPA},
multivariate regular variation under coordinatewise scaling, and
representations of Pickands-type constants and extremal indices discussed in \cite{debicki2017approximation,hashorva2025cluster}.
These three directions will be considered in forthcoming work.

The paper is organised as follows.  The main text focuses on product
groups.  \Cref{5} presents the main results,
\Cref{19} develops product-coordinate
representations, and \Cref{32} studies the resulting
positive semidefinite product-minimum tail-overlap kernels.
The proofs are given in \Cref{46}.  The Appendix collects the
corresponding identities for general locally compact groups, the
common-carrier moment criterion, and the probability normalisation
associated with a nontrivial character ratio
\(\rho=\chi/\Delta_G\).

\section{Main Results}
\label{5}

We work throughout in the measurable product-action setting introduced
above.  Set
\[
m_\alpha(dr)
:=
\prod_{i=1}^q\alpha_i\frac{dr_i}{r_i}.
\]
Then \(m_\alpha\) is a Haar measure on \(G\) and
\[
\vartheta_\alpha(dr)
=
\chi_\alpha(r)^{-1}m_\alpha(dr).
\]

For a measurable function \(f:E\to[0,\infty)\), define
\begin{equation}
\label{6}
K_f(x)
:=
\int_G f(S_rx)\,\vartheta_\alpha(dr),
\qquad x\in E
\end{equation}
and write \(K_A=K_{\1_A}\) for \(A\in\mathcal E\).
Joint measurability implies that \(K_f\) is measurable.  A change of
variables with respect to \(m_\alpha\), together with
\(\vartheta_\alpha=\chi_\alpha^{-1}m_\alpha\), gives
\begin{equation}
\label{7}
K_f(S_sx)
=
\chi_\alpha(s)K_f(x),
\qquad s\in G,\quad x\in E.
\end{equation}

We next introduce the measurable orbit coordinates used in the
existence criterion.  A set \(A\in\mathcal E\) is called
\(\nu\)-conull if \(\nu(E\setminus A)=0\).

\begin{definition}
\label{8}
A measurable map
\[
\tau=(\tau_1,\ldots,\tau_q):
E\longrightarrow[0,\infty)^q
\]
is called a \emph{product gauge} if
\begin{equation}
\label{9}
\tau(S_rx)=r\tau(x),
\qquad r\in G,\quad x\in E.
\end{equation}
It is \emph{admissible for \(\nu\)} if
\[
E_\tau
:=
\bigcap_{i=1}^q\{\tau_i>0\}
\]
is \(\nu\)-conull.
\end{definition}

We say that the action is \emph{free} on a set \(D\subset E\) if
\[
S_rx=x,\quad x\in D,
\qquad\Longrightarrow\qquad
r=\boldsymbol1=(1,\ldots,1)\in G.
\]
It is \emph{essentially free} with respect to \(\nu\) if it is free
on a \(\nu\)-conull invariant measurable set.  In particular, an
admissible product gauge makes the action essentially free: if
\(x\in E_\tau\) and \(S_rx=x\), then
\[
\tau(x)=\tau(S_rx)=r\tau(x)
\]
and the strict positivity of all coordinates of \(\tau(x)\) implies
\(r=\boldsymbol1\).

The set \(E_\tau\) is measurable and \(G\)-invariant, and
\(\tau|_{E_\tau}\) takes values in \(G\).  Throughout the paper,
invariance of a subset of \(E\) means \(G\)-invariance.  Allowing zero
coordinates outside \(E_\tau\) makes exact equivariance possible on
all of \(E\) without imposing conditions on null orbits.

\begin{theorem}
\label{10}
If  \(\nu\in M_{\chi_\alpha,\sigma}(E)\), then   the following are equivalent:
\begin{enumerate}[\rm(i)]
\item
There exists an \(E\)-valued random element \(Z\) such that $\nu=\nu_Z$. 

\item
There exists a measurable function \(f:E\to(0,\infty)\) such that
\begin{equation}
\label{11}
0<\int_E f(x)\,d\nu(x)<\infty,
\qquad
0<K_f<\infty
\quad \nu\text{-almost everywhere}.
\end{equation}

\item
The measure \(\nu\) admits an admissible product gauge.

\item
For every \(i=1,\ldots,q\), there exist measurable functions
\[
g_{n,i}:E\longrightarrow[0,\infty),
\qquad n\ge1
\]
such that
\[
g_{n,i}(S_rx)=r_i g_{n,i}(x),
\qquad r\in G,\quad x\in E,
\]
and
\[
\bigcap_{i=1}^q
\bigcup_{n\ge1}\{g_{n,i}>0\}
\]
is \(\nu\)-conull.
\end{enumerate}
\end{theorem}

The theorem answers the existence question within the
\(\sigma\)-finite class.  Condition \textup{(ii)} is an intrinsic
full-product orbit-integral criterion.  Condition \textup{(iv)} gives
a coordinatewise criterion in terms of equivariant measurable
functions and is often easy to verify on path and sequence spaces.
The functions \(g_{n,i}\) need neither integrability nor continuity.
 
Suppose now that \((E,\mathcal E)\) is standard Borel.  With the
convention
\[
(\nu\circ r)(A):=\nu(S_rA)
\]
the Radon--Nikodym cocycle is the constant function
\[
\frac{d(\nu\circ r)}{d\nu}(x)
=
\chi_\alpha(r)^{-1}.
\]
Consequently, for every strictly positive \(f\in L^1(\nu)\)
\[
K_f(x)
=
\int_G
\frac{d(\nu\circ r)}{d\nu}(x)
f(S_rx)\,m_\alpha(dr)
\]
is the Hopf average of \cite{AvrahamReemPeterzil2026}.  Their
Theorem~A identifies \(\{K_f=\infty\}\) and \(\{K_f<\infty\}\),
modulo \(\nu\), with the conservative and dissipative parts,
respectively.  Thus condition \textup{(ii)} is precisely total
dissipativity in the standard-Borel setting.

Under the logarithmic identification \(G\simeq\mathbb R^q\),
\cite[Thm.~C and Cor.~6.1]{AvrahamReemPeterzil2026} further identifies
total dissipativity with essential freeness and essential smoothness
of the orbit relation.  On a conull invariant Borel set, a Borel
transversal then yields a measurable orbit coordinate and hence a
product gauge.  The standard-Borel case therefore has a dynamical
interpretation in terms of total dissipativity.  The equivalences in
\Cref{10}, however, are formulated directly and hold on an
arbitrary measurable space.

For \(q=1\), related set-based and transversal criteria are given in
\cite[Prop.~2.8, Thm.~2.15]{EvansMolchanov2018}.  These results yield
scalar polar representations.  For \(q>1\), condition \textup{(iv)}
of \Cref{10} requires, for each coordinate \(i\), a countable
family of \(i\)-equivariant measurable functions whose positive sets
have a common \(\nu\)-conull carrier.  This coordinatewise condition
controls the additional product directions.
 
The orbit-integral criterion is pointwise in \(x\); it does not require
\(K_f\) to be integrable with respect to \(\nu\).  Indeed, whenever
\[
0<\int_Ef(x)\,d\nu(x)<\infty
\]
Tonelli's theorem and homogeneity give
\[
\int_EK_f(x)\,\nu(dx)
=
m_\alpha(G)\int_Ef(x)\,d\nu(x)
=
\infty.
\]
If \(Z\) is a stochastic representer, the finite identity used instead
is
\[
\E{K_f(Z)}
=
\int_Ef(x)\,d\nu(x).
\]

The corresponding probability normalisation for a general locally
compact scaling group with nontrivial character ratio is recorded in
\Cref{55}.

\begin{corollary}
\label{12}
Let \(\nu\in M_{\chi_\alpha,\sigma}(E)\).  Then the following
conditions are equivalent.
\begin{enumerate}[\rm(i)]
\item
The measure \(\nu\) is stochastically representable.

\item
For every \(A\in\mathcal E\) with \(\nu(A)<\infty\)
\[
K_A<\infty
\quad \nu\text{-almost everywhere}.
\]

\item
There exists an increasing sequence \((E_n)_{n\ge1}\) in
\(\mathcal E\) such that
\[
E_n\uparrow E,
\qquad
\nu(E_n)<\infty,
\qquad
K_{E_n}<\infty
\quad \nu\text{-almost everywhere}
\]
for every \(n\ge1\).
\end{enumerate}
\end{corollary}

Conditions \textup{(ii)} and \textup{(iii)} of
\Cref{12} are set-based formulations of
\eqref{11}.  They are distinct from the
upper-set covering assumption in
\cite[Prop.~2.8]{EvansMolchanov2018}.

Once \(Z\) exists, Tonelli's theorem gives, for every nonnegative
measurable function \(h\) 
\begin{equation}
\label{13}
\int_Eh(x)\,d\nu(x)
=
\E*{\int_Gh(S_rZ)\,\vartheta_\alpha(dr)}.
\end{equation}
For \(q=1\), this identity is implicit in the push-forward
representations of
\cite[Prop.~2.8, Thm.~2.15]{EvansMolchanov2018}.  Here the radial
measure is the full product measure \(\vartheta_\alpha\); the main
issue is the existence of \(Z\).

For \(q>1\), every group direction must be controlled.  Condition
\textup{(ii)} of \Cref{10} uses the full product group, whereas
condition \textup{(iv)} controls these directions coordinatewise.  A
condition based only on \(\chi_\alpha(r)\) cannot detect the
noncompact directions of \(\ker\chi_\alpha\); see
\Cref{3}.

For scalar tail measures, countable homogeneous coordinates and
stochastic representers are treated in
\cite[Sec.~3.3]{BladtHashorvaShevchenko2022}.  That theory includes
coordinate tail-mass assumptions.  The product example below requires
of \(\nu\) only \(\sigma\)-finiteness and homogeneity.

\begin{example}
\label{14}
Let \(I\) be nonempty and countable, let
\(E=(\mathbb R^q)^I\) have its product sigma-field, and let
\[
(S_rx)_{j,i}=r_i x_{j,i},
\qquad j\in I,\quad i=1,\ldots,q.
\]
Every \(\nu\in M_{\chi_\alpha,\sigma}(E)\) is stochastically
representable.  Choose a sequence \((j_n)\) whose range is \(I\) and
put
\[
g_{n,i}(x)=|x_{j_n,i}|,
\qquad n\ge1,\quad i=1,\ldots,q.
\]
Then \(g_{n,i}(S_rx)=r_i g_{n,i}(x)\).
For fixed \(i\), the set
\[
F_i:=\{x:x_{j,i}=0\text{ for all }j\in I\}
\]
is fixed pointwise by
\(S_{r^{(i,c)}}\), where
\(r^{(i,c)}=(1,\ldots,1,c,1,\ldots,1)\) and
\(c\in(0,\infty)\setminus\{1\}\).  Since
\(\nu\) is \(\sigma\)-finite, write \(F_i=\bigcup_nD_n\) with
\(\nu(D_n)<\infty\).  For every \(n\), homogeneity gives
\[
\nu(D_n)=c^{-\alpha_i}\nu(D_n)
\]
and hence \(\nu(D_n)=0\).  Thus \(\nu(F_i)=0\), and condition
\textup{(iv)} of \Cref{10} follows.
\end{example}

For a \(\sigma\)-finite generated measure, the law of a representer is
not unique.  What is intrinsic is its expectation against every
multihomogeneous functional.

\begin{proposition}
\label{15}
Let \(Z\) and \(W\) be \(E\)-valued random elements. If  
\(\nu_Z \in M_{\chi_\alpha,\sigma}(E)\),  then
\begin{equation}
\label{16}
\nu_W=\nu_Z
\quad\Longleftrightarrow\quad
\E{\Gamma(W)}=\E{\Gamma(Z)}
\quad\text{for every }\Gamma\in\mathcal H_\alpha(E),
\end{equation}
where equality is understood in \([0,\infty]\).
\end{proposition}

For \(q=1\), under the scalar tail-measure assumptions of
\cite[Sec.~3.3]{BladtHashorvaShevchenko2022}, the forward
implication is recorded in \cite[Rem. 3.11(ii), Eq. (3.16)]{BladtHashorvaShevchenko2022}.
That reference does not state the converse in
\eqref{16} or the present product-action result.
The proposition is the probability-measure specialisation of
\Cref{53} for
\(G=(0,\infty)^q\).  This group is unimodular and
\(\rho=\chi_\alpha\).  We include a short adapted proof to keep the
\(Z,W\) criterion self-contained.  The converse follows by taking the
orbit functions \(K_A\) as tests.

Let \(\mathsf T\) be a locally compact  second-countable Hausdorff (lcscH) group
with identity \(e_{\mathsf T}\), and let
\[
B:\mathsf T\times E\longrightarrow E,
\qquad
(a,x)\longmapsto B^ax
\]
be a jointly measurable action commuting with product scaling:
\begin{equation}
\label{17}
B^aS_rx=S_rB^ax,
\qquad a\in\mathsf T,\quad r\in G,\quad x\in E.
\end{equation}
A measure \(\nu\) is called \(B\)-invariant if
\((B^a)_\#\nu=\nu\) for every \(a\in\mathsf T\).  The next result is
an immediate consequence of \Cref{15} and the commutation relation
\eqref{17}; it has no separate proof.  Related scalar
spectral-process criteria appear in
\cite[Rem.~2.2(ii), Cor.~2.3]{KumeE}.  In the present notation,
commutation gives
\(\nu_{B^aZ}=(B^a)_\#\nu_Z\), and \Cref{15} supplies the
equivalence with the moment identities.

\begin{corollary}
\label{18}
If  \(\nu=\nu_Z\in M_{\chi_\alpha,\sigma}(E)\), then   the following are
equivalent: 
\begin{enumerate}[\rm(i)]
\item
\(\nu\) is \(B\)-invariant.

\item
For every \(a\in\mathsf T\)
\[
\nu_{B^aZ}=\nu_Z;
\]
equivalently, every random element \(B^aZ\) is another stochastic representer of
\(\nu\).

\item
For every \(a\in\mathsf T\) and every
\(\Gamma\in\mathcal H_\alpha(E)\)
\[
\E{\Gamma(B^aZ)}=\E{\Gamma(Z)}.
\]
\end{enumerate}
\end{corollary}

The conclusion concerns the homogeneous measure generated by the
shifted random element.  It does not assert that \(B^aZ\) and \(Z\)
have the same law: stationarity of \(\nu_Z\) means that all shifted
random elements \(B^aZ\) generate the same measure.

\section{Product-coordinate representations}
\label{19}

Let \(\tau\) be an admissible product gauge for
\(\nu\in M_{\chi_\alpha,\sigma}(E)\) and   define next 
\[
\Sigma_\tau
:=
\{x\in E_\tau:\tau(x)=\boldsymbol1\},
\qquad
\Pi_\tau x
:=
S_{\tau(x)^{-1}}x,
\qquad x\in E_\tau.
\]
Since \(\nu\) is nonzero and \(E_\tau\) is conull,
\(\Sigma_\tau\ne\varnothing\).  Fix \(w_\tau\in\Sigma_\tau\) and
extend \(\Pi_\tau\) measurably to \(E\) by setting
\(\Pi_\tau x=w_\tau\) for \(x\notin E_\tau\).
Then \(\Pi_\tau:E\to\Sigma_\tau\) is measurable.  By  \eqref{9}   
\[
\tau(\Pi_\tau x)=\boldsymbol1,
\qquad
S_{\tau(x)}\Pi_\tau x=x,
\qquad x\in E_\tau
\]
and for \(w\in\Sigma_\tau\) and \(r\in G\)
\[
\tau(S_rw)=r,
\qquad
\Pi_\tau(S_rw)=w.
\]
Consequently, we have 
\[
\Phi_\tau:
\Sigma_\tau\times G\longrightarrow E_\tau,
\qquad
\Phi_\tau(w,r):=S_rw
\]
is a measurable-space isomorphism, with inverse
\[
\Phi_\tau^{-1}(x)=(\Pi_\tau x,\tau(x)).
\]
Thus the action is free on the conull carrier \(E_\tau\).  This
gauge--cross-section correspondence is classical.  The role of
\Cref{10} is to show that these coordinates exist exactly when a
stochastic polar representation exists.  Here the stochastic polar
representation \eqref{2} is a generally nonunique orbit
mixture generated by a probability law, whereas \eqref{22} below is
the gauge-dependent product-coordinate polar factorisation of \(\nu\).

Put
\[
C_\star:=(1,2]^q,
\qquad
d_\alpha:=m_\alpha(C_\star),
\qquad
h_\tau:=d_\alpha^{-1}\1_{\{\tau\in C_\star\}}
\]
and define a measure on \(\Sigma_\tau\) by
\begin{equation}
\label{20}
\varsigma_\tau
:=
(\Pi_\tau)_\#
\left(
\chi_\alpha(\tau)h_\tau\,\nu|_{E_\tau}
\right).
\end{equation}

Under the classical locally compact continuous-action and
global-cross-section assumptions, the structural factorisation follows
from \cite[Thm.~7.5.1]{Wijsman1990}.  The corresponding orbital
factorisation is given in
\cite[Eq.~(8)]{KamiyaTakemuraKuriki2008}.  In the free case its orbital
factor is \(\vartheta_\alpha\) after translating the multiplier
convention.  The scalar counterpart is the factorisation step in
\cite[Prop.~2.8]{EvansMolchanov2018}.  These results do not directly
cover an arbitrary measurable \(E\) with the present product gauge.
We therefore include a short adapted   proof.  

\begin{lemma}
\label{21}
The measure \(\varsigma_\tau\) is nonzero,  \(\sigma\)-finite and
\begin{equation}
\label{22}
(\Phi_\tau^{-1})_\#(\nu|_{E_\tau})
=
\varsigma_\tau\otimes\vartheta_\alpha.
\end{equation}
It is the unique \(\sigma\)-finite measure on \(\Sigma_\tau\) for which
\eqref{22} holds.
\end{lemma}

For any \(E\)-valued random element \(Z\) concentrated on \(E_\tau\) 
put
\[
R_Z:=\tau(Z),
\qquad
Y_Z:=\Pi_\tau Z
\]
and define
\[
\mu_Z(D)
:=
\E{\chi_\alpha(R_Z)\1_{\{Y_Z\in D\}}},
\qquad
D\in\mathcal E|_{\Sigma_\tau}.
\]
\begin{proposition}
\label{23}
If  \(\nu\in M_{\chi_\alpha,\sigma}(E)\)  
admits an admissible product gauge \(\tau\), then   every stochastic
representer of \(\nu\) is concentrated on \(E_\tau\). 
For every \(E\)-valued random element \(Z\) concentrated on \(E_\tau\) 
$
Y_Z\in\Sigma_\tau,   
Z=S_{R_Z}Y_Z
$ almost surely 
and
\begin{align}
\nu_Z(A)
&=
\int_{\Sigma_\tau}K_A(y)\,\mu_Z(dy),
&&A\in\mathcal E,
\label{24}\\
\mu_Z(D)
&=
\int_{E_\tau}
\chi_\alpha(\tau(x))h_\tau(x)
\1_{\{\Pi_\tau x\in D\}}\,\nu_Z(dx),
&&D\in\mathcal E|_{\Sigma_\tau}.
\label{25}
\end{align}
Moreover  we have 
\begin{equation}
\label{26}
\nu_Z=\nu
\quad\Longleftrightarrow\quad
\mu_Z=\varsigma_\tau
\end{equation} and hence for any \(Z,W\) concentrated on \(E_\tau\)  
$
\nu_Z=\nu_W$ is equivalent to$ 
\mu_Z=\mu_W.$
\end{proposition} 

Thus \Cref{15} and \Cref{23} give two complementary answers to
the \(Z,W\) question: one is intrinsic and uses all multihomogeneous
tests; the other uses any chosen product gauge and compares a single
weighted transverse measure.

Independence of the orbit and cross-section coordinates is classical
for decomposable distributions; see
\cite[Thm.~3.1]{KamiyaTakemuraKuriki2008}.  The next corollary is the
direct consequence of \Cref{23} that identifies exactly when such
independence is possible for a probability representer in our setting.
\begin{corollary}
\label{27}
\label{28} 
For the fixed admissible gauge \(\tau\), with
\(C_\star=(1,2]^q\) and \(d_\alpha=m_\alpha(C_\star)\), set  
\begin{equation}
\label{29}
s_\tau
:=
\varsigma_\tau(\Sigma_\tau)
=
\frac{1}{d_\alpha}
\int_{\{\tau\in C_\star\}}
\chi_\alpha(\tau(x))\,\nu(dx)
=
\frac{\nu\{\tau\in C_\star\}}
{\vartheta_\alpha(C_\star)}
\in(0,\infty].
\end{equation}
Every stochastic representer \(Z\) of \(\nu\) satisfies
\[
\E{\chi_\alpha(R_Z)}=s_\tau.
\]
The following are equivalent:
\begin{enumerate}[\rm(i)]
\item
Some stochastic representer \(Z\) has \(R_Z\perp Y_Z\).

\item
\(s_\tau<\infty\).

\item
\(\nu\{\tau\in C_\star\}<\infty\).
\end{enumerate}

If \(s_\tau<\infty\), put
\[
\bar\alpha:=\sum_{i=1}^q\alpha_i,
\qquad
a_\tau:=s_\tau^{1/\bar\alpha}\boldsymbol1 
\]
and  for any stochastic representer \(Z\)  define
\begin{equation}
\label{30}
\frac{d\Pp_Z^\tau}{d\Pp}
:=
\frac{\chi_\alpha(R_Z)}{s_\tau},
\end{equation}
\begin{equation}
\label{31}
Z_\tau^\circ
:=
S_{a_\tau}Y_Z
=
S_{a_\tau R_Z^{-1}}Z.
\end{equation}
Under \(\Pp_Z^\tau\),
\[
Y_Z\sim\frac{\varsigma_\tau}{s_\tau},
\qquad
\tau(Z_\tau^\circ)=a_\tau,
\qquad
\nu_{Z_\tau^\circ}=\nu,
\]
and
\[
\Law(Z_\tau^\circ)
=
(S_{a_\tau})_\#
\left(\frac{\varsigma_\tau}{s_\tau}\right).
\]
Consequently, every stochastic representer \(W\) satisfying
\(\tau(W)=a_\tau\) almost surely has this law.

Finally, for \(Z\) concentrated on \(E_\tau\) with \(R_Z\perp Y_Z\),
\[
\nu_Z=\nu
\quad\Longleftrightarrow\quad
Y_Z\sim\frac{\varsigma_\tau}{s_\tau}
\ \text{ and }\
\E{\chi_\alpha(R_Z)}=s_\tau.
\]
\end{corollary}

The number \(s_\tau\) may depend on the chosen admissible gauge
\(\tau\).  Once \(\tau\) is fixed, however, both \(s_\tau\) and the
weighted transverse measure are determined by \(\nu\) and do not
depend on the chosen representer.  When \(s_\tau<\infty\), the
resulting gauge-normalised representing law is determined by
\(\nu\) and \(\tau\) alone.

\section{Positive definite  kernels from $\nu$}
\label{32}
\label{33}

Tail measures give rise to positive definite kernels in connection with limit theorems, see e.g.,  \cite{BojanPhilippe, kulik:soulier:2020}.  In this section, we derive such kernels starting directly with a measure $\nu$. Specifically, let \(\nu\in M_{\chi_\alpha}(E)\), let \(J\) be a nonempty index set,
and let
\[
C_j\in\mathcal E,
\qquad
p_j:=\nu(C_j)<\infty,
\qquad j\in J.
\]
Define
\begin{equation}
\label{34}
K_{\nu,C}\bigl((r,j),(s,k)\bigr)
:=
\nu(S_rC_j\cap S_sC_k),
\qquad r,s\in G,\quad j,k\in J.
\end{equation}

\begin{lemma}
\label{35}
The function \(K_{\nu,C}\) is a finite positive semidefinite kernel on
\(G\times J\).  It satisfies
\begin{align}
K_{\nu,C}\bigl((ar,j),(as,k)\bigr)
&=
\chi_\alpha(a)^{-1}
K_{\nu,C}\bigl((r,j),(s,k)\bigr),
\label{36}\\
K_{\nu,C}\bigl((r,j),(r,j)\bigr)
&=
\chi_\alpha(r)^{-1}p_j.
\label{37}
\end{align}
If further  \(\mathsf T\) acts on \(J\), with the action denoted by
\((u,j)\mapsto uj\), \(\nu\) is \(B\)-invariant  and  
\[
C_{uj}=B^uC_j,
\qquad u\in\mathsf T,\quad j\in J,
\]
then
\begin{equation}
\label{38}
K_{\nu,C}\bigl((r,uj),(s,uk)\bigr)
=
K_{\nu,C}\bigl((r,j),(s,k)\bigr).
\end{equation}
\end{lemma}
We next apply \Cref{35} to threshold sets generated by
coordinate-equivariant functions.

Let
\[
\gamma_i:E\to[0,\infty),
\qquad i=1,\ldots,q,
\]
be finite-valued measurable functions satisfying
\begin{equation}
\label{39}
\gamma_i(S_rx)=r_i\gamma_i(x).
\end{equation}
For \(h\in\mathsf T\), put
\[
\gamma_{h,i}(x):=\gamma_i(B^{h^{-1}}x).
\]
Given 
\(\boldsymbol h=(h_1,\ldots,h_q)\in\mathsf T^q\) and \(r\in G\) set
\[
A_{\boldsymbol h}(r)
:=
\bigcap_{i=1}^q\{\gamma_{h_i,i}>r_i\},
\qquad
A_{\boldsymbol h}:=A_{\boldsymbol h}(\boldsymbol1),
\qquad
p_{\boldsymbol h}:=\nu(A_{\boldsymbol h}).
\]
For \(u\in\mathsf T\), we use the diagonal notation
\[
u\boldsymbol h:=(uh_1,\ldots,uh_q).
\]
Assume next that \(p_{\boldsymbol h}<\infty\) for every
\(\boldsymbol h\in\mathsf T^q\), and define
\[
K_\nu\bigl((r,\boldsymbol h),(s,\boldsymbol k)\bigr)
:=
\nu\bigl(A_{\boldsymbol h}(r)\cap
         A_{\boldsymbol k}(s)\bigr).
\]

\begin{corollary}
\label{40}
The function \(K_\nu\) is a finite positive semidefinite kernel on
\(G\times\mathsf T^q\) and
\begin{align}
K_\nu\bigl((ar,\boldsymbol h),(as,\boldsymbol k)\bigr)
&=
\chi_\alpha(a)^{-1}
K_\nu\bigl((r,\boldsymbol h),(s,\boldsymbol k)\bigr),
\label{41}\\
K_\nu\bigl((r,\boldsymbol h),(s,\boldsymbol h)\bigr)
&=
p_{\boldsymbol h}\chi_\alpha(r\vee s)^{-1},
\label{42}
\end{align}
where \((r\vee s)_i=\max(r_i,s_i)\).  If \(\nu\) is
\(B\)-invariant, then
\begin{equation}
\label{43}
K_\nu\bigl((r,u\boldsymbol h),(s,u\boldsymbol k)\bigr)
=
K_\nu\bigl((r,\boldsymbol h),(s,\boldsymbol k)\bigr),
\qquad u\in\mathsf T.
\end{equation}

If \(Z\) is any stochastic representer of \(\nu\), then
\begin{equation}
\label{44}
K_\nu\bigl((r,\boldsymbol h),(s,\boldsymbol k)\bigr)
=
\E*{
\prod_{i=1}^q
\left[
\left(\frac{\gamma_{h_i,i}(Z)}{r_i}\right)^{\alpha_i}
\wedge
\left(\frac{\gamma_{k_i,i}(Z)}{s_i}\right)^{\alpha_i}
\right] }.
\end{equation}
\end{corollary}

The first assertions follow directly from
\Cref{35}.  For \(q=1\) the minimum identity in
\eqref{44} is classical.  See
\cite[Eq.~(3)]{Strokorb}.  Applying this calculation separately in
every radial coordinate gives the displayed \(q\)-fold product.  For
completeness its short Tonelli proof is included in
\Cref{46}.

The last formula does not depend on the chosen representer since its
left-hand side is defined by \(\nu\).  When \(\nu\) is \(\sigma\)-finite
this also follows from \Cref{15}.  A common radial coordinate would
produce one minimum over all constraints.  The product radial measure
\(\vartheta_\alpha\) produces instead the product of the \(q\)
coordinatewise minima in \eqref{44}.

\begin{example}
\label{45}
Let \(E=\mathcal C(\mathsf T,\mathbb R^q)\) have the Borel
sigma-field of the compact-open topology.  Set
\[
(S_rf)(t)=r\odot f(t),
\qquad
(B^af)(t)=f(a^{-1}t),
\qquad
\gamma_i(f)=|f_i(e_{\mathsf T})|.
\]
Then
\[
\gamma_{h,i}(f)=|f_i(h)|,
\qquad
A_{\boldsymbol h}(r)
=
\{f:|f_i(h_i)|>r_i,\ i=1,\ldots,q\}.
\]
If \(\nu=\nu_Z\), formula \eqref{44} applies with
\(\gamma_{h,i}(Z)=|Z_i(h)|\).

Suppose that \(q=1\), \(\alpha_1=1\), \(Z(t)\ge0\), and
\(\E{Z(t)}=1\).  The classical spectral construction
\cite{deHaan}
\[
X(t):=\bigvee_{j\ge1}\Gamma_j^{-1}Z^{(j)}(t)
\]
uses i.i.d. copies \((Z^{(j)})_{j\ge1}\) of \(Z\), independent of the
unit-rate Poisson arrival times \((\Gamma_j)_{j\ge1}\).  The process
has unit Fr\'echet margins.
If \(\nu_{B^aZ}=\nu_Z\) for every \(a\in\mathsf T\), then \(X\) is
stationary.  When \(\nu_Z\) is \(\sigma\)-finite, this condition is
equivalent to the conditions of \Cref{18}.  Its tail-correlation
function is the unit-level overlap of its exponent measure and is
\[
\operatorname{TCF}_X(t)
=
\E{Z(e_{\mathsf T})\wedge Z(t)};
\]
see \cite[Eq.~(3)]{Strokorb}.  Both the Poisson maximum and the minimum
formula are classical.  The stationarity condition follows here from
\Cref{18}.  It concerns the generated measure and does not require
the law of \(Z\) to be stationary.
\end{example}

\begin{remark}
For \(z=(r,\boldsymbol h)\), put \(D_z=A_{\boldsymbol h}(r)\),
\(m_z=\nu(D_z)\), and
\(d_\nu(z,w)=\nu(D_z\mathbin{\triangle}D_w)\); with the empty set as
origin, \(d_\nu\) is a measure-definite pseudometric.  On its metric
quotient, Wang's family in \cite{wang2025family} 
\(C_{\nu,H,\theta}(z,w)
=2^{-\theta}\left[(m_z^{2H}+m_w^{2H})^\theta
-d_\nu(z,w)^{2H\theta}\right]\)   is positive semidefinite for
\(H\in(0,1/2]\), \(\theta\in(0,1]\) and
$$K_\nu(z,w)=\frac12(m_z+m_w-d_\nu(z,w))=C_{\nu,1/2,1}(z,w).$$
Moreover if \(az=(ar,\boldsymbol h)\), then
\(C_{\nu,H,\theta}(az,aw)=\chi_\alpha(a)^{-2H\theta}
C_{\nu,H,\theta}(z,w)\).  After division by \(\theta\), its
\(\theta\downarrow0\) boundary is scale invariant and, for \(q=1\),
\(H=1/2\), \(r\ne s\), and a fixed anchor \(h\) with
\(0<p_h<\infty\), the coordinates \(u=r^{-\alpha}\) and
\(v=s^{-\alpha}\) give \(\log\frac{u+v}{|u-v|}\).  Since this boundary
has infinite diagonal whenever \(m_z>0\), it is understood, under the
test-function hypotheses used there, as the covariance of a generalized
rather than pointwise-defined Gaussian field; we do not pursue this
direction here.
\end{remark}

\section{Proofs}
\label{46}
 
\begin{proof}[Proof of \Cref{10}]
Assume first that \(\nu=\nu_Z\).  Since 
\(\nu\) is nonzero and
\(\sigma\)-finite, there is a finite-valued measurable
\(f:E\to(0,\infty)\) such that
\[
0<\int_Ef(x)\,d\nu(x)<\infty
\]
and hence by  
Tonelli's theorem  
\[
\E{K_f(Z)}=\int_Ef(x)\,d\nu(x)<\infty.
\]
The function \(K_f\) is strictly positive, and
\[
D_f:=\{0<K_f(x)<\infty\}
\]
is $G$-invariant by \eqref{7}.  Hence
\(\Pp\{Z\in D_f\}=1\) and invariance yields
\[
\nu(E\setminus D_f)
=
\E*{\int_G
\1_{E\setminus D_f}(S_rZ)\,\vartheta_\alpha(dr)}
=0,
\]
which proves \textup{(i)}\(\Rightarrow\)\textup{(ii)}.

Assume \textup{(ii)},   write \(D=D_f\) and define    
\[
h(x)
:=
\1_D(x)\frac{f(x)}{K_f(x)}.
\]
For any nonnegative measurable \(u\), put
\[
L_u(x):=\int_Gu(S_{r^{-1}}x)\,m_\alpha(dr).
\]
Again 
Tonelli's theorem and \(\chi_\alpha\)-homogeneity imply  for
nonnegative measurable \(u,v\)
\begin{equation}
\label{47}
\int_Eu(x)K_v(x)\,d\nu(x)
=
\int_E(L_u)(x)v(x)\,d\nu(x).
\end{equation}
The inverse-Haar form of \eqref{6} is
\[
K_f(x)
=
\int_G\chi_\alpha(r)f(S_{r^{-1}}x)\,m_\alpha(dr)
\]
and since
\(K_f(S_{r^{-1}}x)=\chi_\alpha(r)^{-1}K_f(x)\)
we obtain
\begin{equation}
\label{48}
L_h(x)=\1_D(x).
\end{equation}
The reciprocal normalisation and the associated probability kernel
are classical in the standard-Borel setting; see \cite[Lem.~3.1(i)]{Kallenberg2014StationaryDensities} and 
\cite[Thm.~A.29 and the proof of
\textup{(iii)}$\Rightarrow$\textup{(iv)}]{AranoIsonoMarrakchi2021}.
For the probability measure
\(\mu=f\nu/\int_E f\,d\nu\), the normalising function in the latter proof
is precisely \(f/K_f\) on \(D\).  The direct calculation above
remains valid on an arbitrary measurable space.

For \(x\in D\) define a probability measure on \(G\) by
\[
\kappa_x(A)
:=
\int_Ah(S_{r^{-1}}x)\,m_\alpha(dr).
\]
This is a measurable probability kernel from \(D\) to \(G\).  Moreover, for every
\(s\in G\)
\begin{equation}
\label{49}
\kappa_{S_sx}(A)=\kappa_x(s^{-1}A).
\end{equation}
For each \(i\) let
\[
c_i(x)
:=
\inf\left\{
u>0:
\kappa_x\{r:r_i\le u\}\ge\frac12
\right\}.
\]
Every \(c_i\) is measurable and takes values in \((0,\infty)\).
Indeed, for fixed \(u>0\), the marginal distribution function
\[
x\longmapsto
\kappa_x\{r:r_i\le u\}
\]
is measurable; as a function of \(u\), it is nondecreasing and
right-continuous, with limits \(0\) at \(0+\) and \(1\) at infinity.
Moreover, for all \(u>0\)
\[
\{x\in D:c_i(x)\le u\}
=
\left\{x\in D:
\kappa_x\{r:r_i\le u\}\ge\frac12\right\},
\]
hence the function \(c_i\) is measurable, while the 
boundary limits of the distribution function show that
\[
0<c_i(x)<\infty,
\qquad x\in D.
\]  
Equation \eqref{49} implies 
\(c_i(S_sx)=s_ic_i(x)\) and thus
\[
\tau(x)
=
\begin{cases}
(c_1(x),\ldots,c_q(x)),&x\in D,\\
\boldsymbol0,&x\notin D
\end{cases}
\]
is an admissible product gauge, which  proves
\textup{(ii)}\(\Rightarrow\)\textup{(iii)}.

If \textup{(iii)} holds, take
\(g_{1,i}=\tau_i\) and \(g_{n,i}=0\) for \(n\ge2\).
This proves \textup{(iii)}\(\Rightarrow\)\textup{(iv)}.

Conversely, assume \textup{(iv)} and define
\[
B_{n,i}
:=
\{g_{n,i}>0\}
\setminus\bigcup_{k<n}\{g_{k,i}>0\},
\qquad
\tau_i
:=
\sum_{n\ge1}g_{n,i}\1_{B_{n,i}}.
\]
For fixed \(i\), the sets \(B_{n,i}\) are invariant and pairwise
disjoint.  Hence \(\tau_i\) is finite-valued and measurable 
\[
\tau_i(S_rx)=r_i\tau_i(x),
\qquad
\{\tau_i>0\}=\bigcup_{n\ge1}\{g_{n,i}>0\} 
\]
implying that  \(\tau=(\tau_1,\ldots,\tau_q)\) is admissible establishing 
\textup{(iv)}\(\Rightarrow\)\textup{(iii)}.

It remains to prove \textup{(iii)}\(\Rightarrow\)\textup{(i)}.
Let \(\tau\) be admissible and put
\[
h_\tau
=
d_\alpha^{-1}\1_{\{\tau\in C_\star\}} ,
\qquad
C_\star=(1,2]^q,
\qquad
d_\alpha=m_\alpha(C_\star).
\]
Indeed,  \eqref{9} implies  for every \(x\in E\)
\[
\begin{aligned}
L_{h_\tau}(x)
&=
\frac{1}{d_\alpha}
\int_G
\1_{C_\star}\bigl(\tau(S_{r^{-1}}x)\bigr)\,
m_\alpha(dr)=
\frac{1}{d_\alpha}
\int_G
\1_{C_\star}\bigl(r^{-1}\tau(x)\bigr)\,
m_\alpha(dr).
\end{aligned}
\]
If \(x\in E_\tau\), then \(\tau(x)\in G\).  The change of variables
\(u=r^{-1}\tau(x)\), which preserves the product Haar measure
\(m_\alpha\), yields
\[
L_{h_\tau}(x)
=
\frac{1}{d_\alpha}
\int_G\1_{C_\star}(u)\,m_\alpha(du)
=
\frac{m_\alpha(C_\star)}{d_\alpha}
=
1.
\]
If \(x\notin E_\tau\), then some coordinate of \(\tau(x)\) is zero.
Consequently,
\(r^{-1}\tau(x)\notin C_\star=(1,2]^q\) for every \(r\in G\), and hence
\(L_{h_\tau}(x)=0\) implying  
\begin{equation}
\label{50}
L_{h_\tau}=\1_{E_\tau}.
\end{equation}

Let \(\mu=h_\tau\nu\).  Since \(h_\tau\) is bounded and \(\nu\) is
\(\sigma\)-finite, \(\mu\) is \(\sigma\)-finite.  Moreover,
\eqref{47} and \eqref{50} give, for
every \(A\in\mathcal E\),
\[
\int_EK_A\,d\mu
=
\int_E(L_{h_\tau})\1_A\,d\nu
=
\nu(A),
\]
because \(E_\tau\) is \(\nu\)-conull.  Since \(\nu\) is nonzero,  then so
is \(\mu\). Put
\[
\bar\alpha:=\sum_{i=1}^q\alpha_i,
\qquad
j_\alpha(u)
:=
\bigl(u^{1/\bar\alpha},\ldots,u^{1/\bar\alpha}\bigr),
\qquad u>0,
\]
so that \(\chi_\alpha(j_\alpha(u))=u\).
Choose a finite-valued measurable \(p:E\to(0,\infty)\) with
\[
\int_Ep(x)\,d\mu(x)=1.
\]
Let \(Y\) have law \(p\mu\), and set
\[
Z=S_{j_\alpha(1/p(Y))}Y.
\]
 
For \(A\in\mathcal E\), \eqref{7} implies
\[
\begin{aligned}
\nu_Z(A)
&=\E{K_A(Z)}=\int_Ep(y)
\chi_\alpha(j_\alpha(1/p(y)))K_A(y)\,\mu(dy)=\int_EK_A(y)\,\mu(dy)
=\nu(A)
\end{aligned}
\]
and thus \(Z\) is a stochastic representer of $\nu$ establishing the claim.
\end{proof}

\begin{proof}[Proof of \Cref{12}]
Assume first that \(\nu=\nu_Z\), and let \(A\in\mathcal E\) satisfy
\(\nu(A)<\infty\).  Then
\[
\E{K_A(Z)}=\nu(A)<\infty
\]
and hence \(\Pp\{Z\in D_A\}=1\), where
\[
D_A:=\{x\in E:K_A(x)<\infty\}.
\]
By \eqref{7}, the set \(D_A\) is
invariant (recall that invariant means $G$-invariant).  Consequently, for almost every outcome
\[
\1_{E\setminus D_A}(S_rZ)=0
\qquad\text{for every }r\in G.
\]
It follows directly from the representation that
\[
\nu(E\setminus D_A)
=
\E*{\int_G
\1_{E\setminus D_A}(S_rZ)\,\vartheta_\alpha(dr)}
=0
\]
and hence  \textup{(i)} implies \textup{(ii)}.

Since \(\nu\) is \(\sigma\)-finite, it has an increasing measurable
exhaustion by sets of finite measure and hence \textup{(ii)} implies
\textup{(iii)}.

Assume \textup{(iii)} and choose a finite-valued measurable
\(p:E\to(0,\infty)\) such that
\[
\int_Ep(x)\,d\nu(x)=1
\]
and let
\[
\mathbb P_0(dx):=p(x)\nu(dx).
\]
Both \(\mathbb P_0\) and \(\nu\) have the same null sets.  For each
\(n\), \(K_{E_n}<\infty\) \(\mathbb P_0\)-almost surely.  We may
therefore choose \(c_n>0\) sufficiently small that
\[
c_n\le2^{-n},
\qquad
c_n\nu(E_n)\le2^{-n},
\qquad
\mathbb P_0\{c_nK_{E_n}>2^{-n}\}\le2^{-n}.
\]
Indeed, for fixed \(n\) 
\(\mathbb P_0\{cK_{E_n}>2^{-n}\}\to0\) as \(c\downarrow0\).

Define
\[
f:=\sum_{n\ge1}c_n\1_{E_n}.
\]
Since the sets \(E_n\) cover \(E\) 
\[
0<f\le\sum_{n\ge1}c_n\le1
\quad\text{on }E
\]
and because \(\nu\) is nonzero 
\[
0<\int_Ef(x)\,d\nu(x)
=
\sum_{n\ge1}c_n\nu(E_n)
\le
\sum_{n\ge1}2^{-n}
<\infty.
\]
Moreover, Tonelli's theorem gives
\[
K_f=\sum_{n\ge1}c_nK_{E_n}.
\]
By the Borel--Cantelli lemma 
\[
c_nK_{E_n}\le2^{-n}
\]
for all but finitely many \(n\), \(\mathbb P_0\)-almost surely.  Since
every \(K_{E_n}\) is finite simultaneously outside a
\(\mathbb P_0\)-null set it follows that
\[
K_f<\infty
\quad \mathbb P_0\text{-almost surely},
\]
and hence \(\nu\)-almost everywhere.  Since \(f>0\) everywhere and
\(\vartheta_\alpha\) is nonzero 
 also \(K_f>0\) everywhere.  Thus
\(f\) satisfies condition
\textup{(ii)} of \Cref{10}, which proves that \(\nu\) is
stochastically representable.
\end{proof}

\begin{proof}[Proof of \Cref{15}]
Suppose first that \(\nu_W=\nu_Z=:\nu\).  Choose a finite-valued
measurable \(f_0>0\) such that
\[
c:=\int_Ef_0(x)\,d\nu(x)\in(0,\infty),
\qquad
D:=\{x\in E:0<K_{f_0}(x)<\infty\}.
\]
As in the first part of the preceding proof 
\[
\Pp\{Z\in D\}=\Pp\{W\in D\}=1 
\]
and \(D\) is invariant.  For
\(\Gamma\in\mathcal H_\alpha(E)\)  define
\[
f_\Gamma
:=
\1_D\,\frac{f_0\Gamma}{K_{f_0}},
\]
where this is defined to be zero outside \(D\).
The ratio \(\Gamma/K_{f_0}\) is invariant on \(D\), whence
\[
K_{f_\Gamma}=\1_D\Gamma.
\]
Tonelli's theorem implies   
\[
\E{\Gamma(Z)}
=
\E{K_{f_\Gamma}(Z)}
=
\int_Ef_\Gamma(x)\,d\nu(x)
=
\E{K_{f_\Gamma}(W)}
=
\E{\Gamma(W)}.
\]

Conversely, \(K_A\in\mathcal H_\alpha(E)\) for every
\(A\in\mathcal E\).  Equality of all multihomogeneous moments implies
\[
\nu_Z(A)=\E{K_A(Z)}=\E{K_A(W)}=\nu_W(A),
\]
which proves the assertion.
\end{proof}

\begin{proof}[Proof of \Cref{21}]
From \eqref{50} and
\eqref{47} for every nonnegative measurable \(f\)
\begin{equation}
\label{51}
\begin{aligned}
\int_{\Sigma_\tau}K_f(w)\,\varsigma_\tau(dw)
&=
\int_{E_\tau}
\chi_\alpha(\tau(x))h_\tau(x)
K_f(\Pi_\tau x)\,\nu(dx)\\
&=
\int_Eh_\tau(x)K_f(x)\,\nu(dx)\\
&=
\int_Ef(x)\,d\nu(x).
\end{aligned}
\end{equation}
Here we used
\(K_f(\Pi_\tau x)=
\chi_\alpha(\tau(x))^{-1}K_f(x)\).

To see directly that \(\varsigma_\tau\) is \(\sigma\)-finite choose
\(E_n\uparrow E\) with \(\nu(E_n)<\infty\) and put
\[
D_{n,k}
:=
\{w\in\Sigma_\tau:K_{E_n}(w)\ge1/k\},
\qquad n,k\ge1.
\]
Since \(K_{E_n}(w)\uparrow K_E(w)=\vartheta_\alpha(G)=\infty\)  the
sets \(D_{n,k}\) cover \(\Sigma_\tau\).  Moreover, 
\eqref{51} yields 
\[
\varsigma_\tau(D_{n,k})
\le
k\int_{\Sigma_\tau}K_{E_n}(x)\,d\varsigma_\tau(x)
=
k\nu(E_n)
<\infty
\]
implying that   \(\varsigma_\tau\) is \(\sigma\)-finite.  It is nonzero because 
for a set \(A_\star\) with \(0<\nu(A_\star)<\infty\) we have 
\[
\int_{\Sigma_\tau}K_{A_\star}(w)\,\varsigma_\tau(dw)
=
\nu(A_\star)
>0.
\]

Now let \(F\) be nonnegative and measurable on
\(\Sigma_\tau\times G\) and define
\[
f_F(x)
:=
\begin{cases}
F(\Pi_\tau x,\tau(x)),&x\in E_\tau,\\
0,&x\notin E_\tau.
\end{cases}
\]
For \(w\in\Sigma_\tau\), by    \eqref{9}  
\[
\tau(S_rw)=r,
\qquad
\Pi_\tau(S_rw)=w
\]
and therefore
\[
K_{f_F}(w)
=
\int_G F(w,r)\,\vartheta_\alpha(dr).
\]
Applying \eqref{51} to \(f_F\) now yields
\[
\int_{E_\tau}F(\Pi_\tau x,\tau(x))\,\nu(dx)
=
\int_{\Sigma_\tau}\int_G
F(w,r)\,\vartheta_\alpha(dr)\,\varsigma_\tau(dw),
\]
which is precisely \eqref{22}.  Finally, if another
\(\sigma\)-finite measure \(\varsigma'\) satisfies \eqref{22},
evaluate both product measures on \(D\times C\), where
\(D\in\mathcal E|_{\Sigma_\tau}\) and
\(0<\vartheta_\alpha(C)<\infty\).  Division by
\(\vartheta_\alpha(C)\) gives
\(\varsigma'(D)=\varsigma_\tau(D)\).
\end{proof}

\begin{proof}[Proof of \Cref{23}]
Suppose that \(Z\) represents \(\nu\).  Since \(E_\tau\) is $G$-invariant 
the orbit integral of \(\1_{E\setminus E_\tau}\) equals
\(\vartheta_\alpha(G)=\infty\) on the event
\(\{Z\notin E_\tau\}\).  Since
\(\nu_Z(E\setminus E_\tau)=\nu(E\setminus E_\tau)=0\) this event has
probability zero.  Thus every representer belongs to \(E_\tau\) almost
surely.

Now suppose that \(\Pp\{Z\in E_\tau\}=1\).  In view of \eqref{9} we have 
\(Y_Z\in\Sigma_\tau\) and \(Z=S_{R_Z}Y_Z\) almost surely.  Hence for
\(A\in\mathcal E\) 
\[
\nu_Z(A)
 =
\E{K_A(Z)}
 =
\E{\chi_\alpha(R_Z)K_A(Y_Z)}
 =
\int_{\Sigma_\tau}K_A(y)\,\mu_Z(dy),
\]
which proves \eqref{24}.

For \(D\in\mathcal E|_{\Sigma_\tau}\) define on \(E_\tau\)
\[
g_D(x)
:=
\chi_\alpha(\tau(x))h_\tau(x)
\1_{\{\Pi_\tau x\in D\}}
\]
and extend it by zero outside \(E_\tau\).  For
\(y\in\Sigma_\tau\)
\[
\begin{aligned}
K_{g_D}(y)
&=
\frac{\1_D(y)}{d_\alpha}
\int_{C_\star}\chi_\alpha(r)\,\vartheta_\alpha(dr)=
\frac{m_\alpha(C_\star)}{d_\alpha}\1_D(y)
=
\1_D(y).
\end{aligned}
\]
Since \(Z=S_{R_Z}Y_Z\), the orbit-scaling identity
\eqref{7} gives
\[
K_{g_D}(Z)
=
K_{g_D}(S_{R_Z}Y_Z)
=
\chi_\alpha(R_Z)K_{g_D}(Y_Z)
=
\chi_\alpha(R_Z)\1_{\{Y_Z\in D\}}.
\]
Therefore Tonelli's theorem gives
\[
\int_Eg_D(x)\,d\nu_Z(x)
=
\E{K_{g_D}(Z)}
=
\E{\chi_\alpha(R_Z)\1_{\{Y_Z\in D\}}}
=
\mu_Z(D),
\]
which proves \eqref{25}.

If \(\nu_Z=\nu\), then \eqref{25} and
\eqref{20} give \(\mu_Z=\varsigma_\tau\).
Conversely, if \(\mu_Z=\varsigma_\tau\), then
\eqref{24} and
\eqref{51} give \(\nu_Z=\nu\).  Finally,
\eqref{24} proves
\(\mu_Z=\mu_W\Rightarrow\nu_Z=\nu_W\), while
\eqref{25} proves the converse.  Neither implication
requires the generated measures to be \(\sigma\)-finite.
\end{proof}

\begin{proof}[Proof of \Cref{28}]
The first identity in \eqref{29} follows from
\eqref{20}; the second follows from
\eqref{22}.  Evaluating \eqref{26} at
\(D=\Sigma_\tau\) gives
\[
\E{\chi_\alpha(R_Z)}=s_\tau
\]
for every representer \(Z\).

If \(R_Z\) and \(Y_Z\) are independent, \eqref{26} yields
\[
\varsigma_\tau(D)
=
\E{\chi_\alpha(R_Z)}\Pp\{Y_Z\in D\}.
\]
The expectation cannot be infinite: since \(\varsigma_\tau\) is
\(\sigma\)-finite, a countable cover of \(\Sigma_\tau\) by sets of
finite \(\varsigma_\tau\)-mass contains one having positive
\(Y_Z\)-probability, which would give a contradiction.  Hence
\[
s_\tau=\E{\chi_\alpha(R_Z)}<\infty,
\qquad
Y_Z\sim\varsigma_\tau/s_\tau.
\]
Conversely, if \(s_\tau<\infty\), choose
\(Y_Z\sim\varsigma_\tau/s_\tau\) and, independently, any \(G\)-valued
\(R_Z\) satisfying \(\E{\chi_\alpha(R_Z)}=s_\tau\).  Such a choice
exists because \(\chi_\alpha(G)=(0,\infty)\).  Then
\(Z=S_{R_Z}Y_Z\) satisfies \eqref{26}.

Now fix any stochastic representer \(Z\).  Since
\(\E{\chi_\alpha(R_Z)}=s_\tau<\infty\),
\eqref{30} defines a probability measure.  For every
\(D\in\mathcal E|_{\Sigma_\tau}\), \eqref{26} gives
\[
\Pp_Z^\tau\{Y_Z\in D\}
=
\frac{1}{s_\tau}
\E{\chi_\alpha(R_Z)\1_{\{Y_Z\in D\}}}
=
\frac{\varsigma_\tau(D)}{s_\tau}.
\]
Under \(\Pp_Z^\tau\) using \eqref{9} implies  
\[
\tau(Z_\tau^\circ)=a_\tau,
\qquad
\Pi_\tau Z_\tau^\circ=Y_Z
\quad\text{almost surely}.
\]
Therefore the weighted transverse measure of \(Z_\tau^\circ\) is
\[
\chi_\alpha(a_\tau)
\Pp_Z^\tau\{Y_Z\in D\}
=
s_\tau\frac{\varsigma_\tau(D)}{s_\tau}
=
\varsigma_\tau(D).
\]
By \Cref{23}, \(Z_\tau^\circ\) represents \(\nu\).  Its law is
\[
\bigl(y\mapsto S_{a_\tau}y\bigr)_\#
\left(\frac{\varsigma_\tau}{s_\tau}\right),
\]
which depends only on \(\nu\) and \(\tau\), and not on \(Z\).
If \(W\) represents \(\nu\) and
\(\tau(W)=a_\tau\) almost surely, then \eqref{26} yields
\[
\Pp\{\Pi_\tau W\in D\}
=
\frac{\varsigma_\tau(D)}{s_\tau}.
\]
Since \(W=S_{a_\tau}\Pi_\tau W\) almost surely, \(W\) has the same
law as \(Z_\tau^\circ\), proving uniqueness in law.  Finally,
\(0<\vartheta_\alpha(C_\star)<\infty\), so
\eqref{29} proves the equivalence of
\textup{(ii)} and \textup{(iii)}.
\end{proof}

\begin{proof}[Proof of \Cref{35}] 

Finiteness follows from
\[
0\le K_{\nu,C}((r,j),(s,k))
\le
\min\{\chi_\alpha(r)^{-1}p_j,
      \chi_\alpha(s)^{-1}p_k\}.
\]
For \(z_\ell=(r_\ell,j_\ell)\) and \(c_\ell\in\mathbb C\) 
\[
\sum_{\ell,m}c_\ell\overline{c_m}
K_{\nu,C}(z_\ell,z_m)
=
\int_E
\left|
\sum_\ell c_\ell\1_{S_{r_\ell}C_{j_\ell}}(x)
\right|^2\nu(dx)
\ge0.
\]
Relative invariance under a common product scale gives
\eqref{36} and
\eqref{37}.   
Finally, equivariance and the commutation of \(B\) with \(S\) give
\[
S_rC_{uj}\cap S_sC_{uk}
=
B^u\bigl(S_rC_j\cap S_sC_k\bigr).
\]
Since \(B^u\) is a measurable bijection with inverse \(B^{u^{-1}}\),
\(B\)-invariance gives \(\nu(B^uA)=\nu(A)\) for every measurable
\(A\).  Hence it yields
\eqref{38}.
\end{proof}

\begin{proof}[Proof of \Cref{40}]
The kernel assertions follow from \Cref{35}; only the
product-radial minimum calculation remains.

Covariance and commutation imply 
\[
A_{\boldsymbol h}(r)=S_rA_{\boldsymbol h},
\qquad
A_{u\boldsymbol h}=B^uA_{\boldsymbol h}.
\]
Thus \Cref{35} proves positivity,
\eqref{41} and \eqref{43}.  Moreover, we have 
\[
A_{\boldsymbol h}(r)\cap A_{\boldsymbol h}(s)
=
A_{\boldsymbol h}(r\vee s),
\]
which implies  \eqref{42}.

If \(\nu=\nu_Z\), Tonelli's theorem reduces the radial integral to
\(q\) one-dimensional integrals.  For each \(i\) 
\[
\int_0^\infty
\1_{\{t_i\gamma_{h_i,i}(Z)>r_i\}}
\1_{\{t_i\gamma_{k_i,i}(Z)>s_i\}}
\alpha_it_i^{-\alpha_i-1}\,dt_i
=
\left(\frac{\gamma_{h_i,i}(Z)}{r_i}\right)^{\alpha_i}
\wedge
\left(\frac{\gamma_{k_i,i}(Z)}{s_i}\right)^{\alpha_i}.
\]
Multiplication and expectation prove
\eqref{44}.
\end{proof}

\section{Appendix}
\label{52}
Let \(G\) be a lcscH  group, let
\(m_G\) be a left Haar measure with
\[
m_G(Ag)=\Delta_G(g)m_G(A)
\]
and let \(\chi:G\to(0,\infty)\) be a continuous character.  Put
\[
\vartheta_\chi(dg)=\chi(g)^{-1}m_G(dg),
\qquad
\rho(g)=\frac{\chi(g)}{\Delta_G(g)}.
\]
For a jointly measurable left action \(S\) on \((E,\mathcal E)\) set
\[
K_f(x)=\int_Gf(S_gx)\,\vartheta_\chi(dg),\qquad
L_h(x)=\int_Gh(S_{g^{-1}}x)\,m_G(dg)
\]
and for a positive measure \(\mu\)
\[
(\mathsf K^\ast\mu)(A)=\int_EK_A(x)\,d\mu(x).
\]
A measure \(\nu\) is \(\chi\)-homogeneous if
\[
\nu(S_gA)=\chi(g)^{-1}\nu(A),
\qquad g\in G,\quad A\in\mathcal E,
\]
and \(\nu\in M_{\chi,\sigma}(E)\), if it is also nonzero and
\(\sigma\)-finite.  A positive \(\mu\) satisfying
\(\mathsf K^\ast\mu=\nu\) is a \emph{representing measure}; for an
\(E\)-valued random element \(Z\) with law $\pi_Z$ write
\[
\nu_Z(A)=\E{K_A(Z)}
=(\mathsf K^\ast\pi_Z)(A).
\]
Finally, let
\[
\mathcal H_\rho(E)
=
\{\Gamma:E\to[0,\infty]\text{ measurable}:
\Gamma(S_gx)=\rho(g)\Gamma(x),\ g\in G,\ x\in E\}.
\]

For \(b\in G\), any positive measure \(\eta\), and nonnegative
measurable \(f,h\), right translation, monotone approximation and
Tonelli's theorem give
\[
K_f(S_bx)=\rho(b)K_f(x),
\]
\[
\mathsf K^\ast\eta=\nu
\ \Longleftrightarrow\
\int_EK_f(x)\,d\eta(x)=\int_Ef(x)\,d\nu(x)\quad\text{for every }f\ge0,
\]
and whenever \(\nu\) is \(\chi\)-homogeneous and \(\sigma\)-finite
\[
\int_Eh(x)K_f(x)\,d\nu(x)=\int_E(L_h(x))f(x)\,d\nu(x),
\qquad
\mathsf K^\ast(h\nu)=(L_h)\nu.
\]
For \(\chi=1\), the representation
\(\mathsf K^\ast((w/K_w)\nu)=\nu\), for measurable \(w>0\) with
\(K_w<\infty\) everywhere, is the self-mixture formula of
\cite[Prop.~3.3]{Kallenberg2014StationaryDensities}.
The cited proposition applies to \(s\)-finite invariant measures
on measurable spaces under this properness assumption.

The scalar forward identity below appears in
\cite[Rem.~3.11(ii), Eq. (3.16)]
{BladtHashorvaShevchenko2022}; compare also
\cite[Thm.~2.4]{Kallenberg2007Invariant} and
\cite{Kallenberg2007InvariantErratum} for proper Borel actions.

\begin{lemma}
\label{53}
If \(\nu\in M_{\chi,\sigma}(E)\) and  
\(\mathsf K^\ast\mu=\nu\) for some positive measure \(\mu\), then \(\mu\) is nonzero,  \(\sigma\)-finite and for every positive measure \(\lambda\)
\begin{equation}
\label{54}
\mathsf K^\ast\lambda=\nu
\quad\Longleftrightarrow\quad
\int_E\Gamma(x)\,d\lambda(x)=\int_E\Gamma(x)\,d\mu(x)
\quad\text{for every }
\Gamma\in\mathcal H_\rho(E).
\end{equation}
\end{lemma}

\begin{proof}
Choose finite-valued \(f_0>0\) with
\[
c:=\int_Ef_0(x)\,d\nu(x)\in(0,\infty),
\qquad
D:=\{x\in E: 0<K_{f_0}(x)<\infty\}.
\]
Every \(\eta\) satisfying \(\mathsf K^\ast\eta=\nu\) obeys
\[
\int_EK_{f_0}(x)\,d\eta(x)=c
\]
and hence  \(D\) is \(\eta\)-conull.  It is further invariant and consequently
\[
\nu(E\setminus D)=\int_EK_{E\setminus D}(x)\,d\eta(x)=0.
\]
Since \(\mathsf K^\ast\eta=\nu\ne0\), every representing measure
\(\eta\) is nonzero.  Moreover we have 
\[
D_n:=D\cap\{x\in E: K_{f_0}(x)\ge1/n\},
\qquad
\eta(D_n)\le nc
\]
and therefore  every representing measure is \(\sigma\)-finite.

For \(\Gamma\in\mathcal H_\rho(E)\) define  with value zero off \(D\) 
\[
f_\Gamma=\1_D\frac{f_0\Gamma}{K_{f_0}}.
\]
The ratio \(\Gamma/K_{f_0}\) is invariant on \(D\), whence
\[
K_{f_\Gamma}=\1_D\Gamma
\quad\text{and}\quad
\int_E\Gamma(x)\,d\eta(x)=\int_Ef_\Gamma(x)\,d\nu(x)
\]
for every representing measure \(\eta\).  This proves the forward
implication in \eqref{54}; the converse follows
by taking \(\Gamma=K_A\), \(A\in\mathcal E\).
\end{proof}

\begin{lemma}
\label{55}
Let \(\nu\in M_{\chi,\sigma}(E)\) and consider:
\begin{enumerate}[\rm(i)]
\item
\(\nu=\nu_Z\) for some \(E\)-valued random element \(Z\).

\item
There is a finite-valued measurable \(h:E\to[0,\infty)\) such that
\[
L_h=1\quad\nu\text{-almost everywhere}.
\]

\item
There is a finite-valued measurable \(f:E\to(0,\infty)\) such that
\[
0<K_f<\infty\quad\nu\text{-almost everywhere}.
\]
\end{enumerate}
Conditions \textup{(ii)} and \textup{(iii)} are always equivalent.
If \(\rho\not\equiv1\), all three conditions are equivalent.

If \(\rho\equiv1\), condition \textup{(i)} holds if and only if
there is an \(h\) as in \textup{(ii)} satisfying
\[
\int_Eh(x)\,d\nu(x)=1.
\]
Whenever positive representing measures exist, they have a common
total mass, and a probability representing measure exists precisely
when this common mass is one.
\end{lemma}

\begin{proof}
Assume \textup{(iii)} and put
\[
D_f=\{x\in E: 0<K_f(x)<\infty\},
\qquad
h_f(x)=\1_{D_f}\frac{f(x)}{K_f(x)},
\]
with \(h_f=0\) off \(D_f\).  
Haar inversion gives
\[
K_f(x)
=
\int_G\rho(g)f(S_{g^{-1}}x)\,m_G(dg).
\]
Moreover, \(D_f\) is invariant and
\[
K_f(S_{g^{-1}}x)=\rho(g)^{-1}K_f(x).
\]
Therefore
\[
\begin{aligned}
L_{h_f}(x)
&=
\int_G
\1_{D_f}(S_{g^{-1}}x)
\frac{f(S_{g^{-1}}x)}
     {K_f(S_{g^{-1}}x)}
\,m_G(dg)\\
&=
\frac{\1_{D_f}(x)}{K_f(x)}
\int_G\rho(g)f(S_{g^{-1}}x)\,m_G(dg)
=
\1_{D_f}(x).
\end{aligned}
\]
Thus \textup{(ii)} follows. 
Conversely, under \textup{(ii)} 
\[
\mu:=h\nu,
\qquad
\mathsf K^\ast\mu=(L_h)\nu=\nu.
\]
The common-carrier construction in the proof of
\Cref{53} then gives
\textup{(iii)}.  Applied with \(\mu=\pi_Z\)  it also proves
\textup{(i)}\(\Rightarrow\)\textup{(iii)}.

Suppose now that \(\rho\not\equiv1\).  Choose \(b\in G\), replacing it
by its inverse if necessary, such that \(a:=\rho(b)>1\).  For \(u>0\),
put
\[
n(u)=\left\lfloor\frac{\log u}{\log a}\right\rfloor,\qquad
\lambda(u)=
\frac{a^{n(u)+1}-u}{a^{n(u)+1}-a^{n(u)}},
\]
\[
Q_u=\lambda(u)\delta_{b^{n(u)}}
 +(1-\lambda(u))\delta_{b^{n(u)+1}}.
\]
Then \(Q\) is a probability kernel satisfying
\[
\int_G\rho(g)\,Q_u(dg)=u.
\]
For the representing measure \(\mu=h\nu\), choose finite-valued
\(p>0\) with \(\int_Ep\,d\mu=1\), and let
\[
\Pp\{Y\in dx,R\in dg\}
=p(x)\mu(dx)Q_{1/p(x)}(dg).
\]
With \(Z=S_RY\) covariance gives
\[
\E{K_A(Z)}
=
\int_Ep(x)K_A(x)
 \left(\int_G\rho(g)Q_{1/p(x)}(dg)\right)\mu(dx)
=
\nu(A).
\]
Thus \textup{(ii)} implies \textup{(i)}.

Finally, let \(\rho\equiv1\).  If \(h\) satisfies \textup{(ii)} and
\(\int_Eh\,d\nu=1\), then \(h\nu\) is a probability representing
measure.  Conversely, if \(\nu=\nu_Z\), the preceding construction
gives \(h\) with \(L_h=1\), so \(h\nu\) is representing.  Since
\(1\in\mathcal H_\rho(E)\),
\[
\int_Eh(x)\,d\nu(x)=\int_E1\,d\pi_Z(x) =1
\]
by \Cref{53}.  Applying the same criterion
to any two positive representing measures shows that their total
masses agree.  Hence random shifts cannot change this mass, and a
probability solution exists exactly when it equals one.
\end{proof}

\bibliographystyle{ieeetr}
\bibliography{EEEA,EEEA_V42_additions,EEEA_V44_additions,EEEA_V48_additions,EB}

\newcommand{\nosort}[1]{}
\begin{thebibliography}{10}

\bibitem{EvansMolchanov2018}
S.~N. Evans and I.~Molchanov, ``Polar decomposition of scale-homogeneous measures with application to {L}\'{e}vy measures of strictly stable laws,'' {\em J. Theoret. Probab.}, vol.~31, no.~3, pp.~1303--1321, 2018.

\bibitem{wao}
T.~Owada and G.~Samorodnitsky, ``Tail measures of stochastic processes or random fields with regularly varying tails.'' Technical report, 2012.

\bibitem{MR3561100}
G.~Samorodnitsky, {\em Stochastic processes and long range dependence}.
\newblock Springer Series in Operations Research and Financial Engineering, Springer, Cham, 2016.

\bibitem{DombryHashorvaSoulier2018}
C.~Dombry, E.~Hashorva, and P.~Soulier, ``Tail measure and spectral tail process of regularly varying time series,'' {\em Ann. Appl. Probab.}, vol.~28, no.~6, pp.~3884--3921, 2018.

\bibitem{Hrovje}
H.~Planini\'{c} and P.~Soulier, ``The tail process revisited,'' {\em Extremes}, vol.~21, no.~4, pp.~551--579, 2018.

\bibitem{PH2020}
P.~Soulier, ``The tail process and tail measure of continuous time regularly varying stochastic processes,'' {\em Extremes}, vol.~25, no.~1, pp.~107--173, 2022.

\bibitem{BladtHashorvaShevchenko2022}
M.~Bladt, E.~Hashorva, and G.~Shevchenko, ``Tail measures and regular variation,'' {\em Electron. J. Probab.}, vol.~27, pp.~Paper No. 64, 43, 2022.

\bibitem{Guenter}
G.~Last, ``Tail processes and tail measures: An approach via palm calculus,'' {\em Extremes}, vol.~26, no.~4, pp.~715--746, 2023.

\bibitem{Resnickart}
S.~Resnick, {\em The Art of Finding Hidden Risks: Hidden Regular Variation in the 21st Century}.
\newblock Springer Nature, 2024.

\bibitem{Ilya25}
B.~Basrak, N.~Milinčević, and I.~Molchanov, ``Foundations of regular variation on topological spaces,'' {\em arXiv preprint arXiv:2503.00921}, 2025.

\bibitem{Hashorva2026ShiftGenerated}
E.~Hashorva, ``Shift-generated classes of jointly measurable random fields,'' {\em Lithuanian Math. J.}, 2026.
\newblock To appear.

\bibitem{AvrahamReemPeterzil2026}
N.~Avraham-Re'em and G.~Peterzil, ``The {H}opf decomposition of locally compact group actions,'' {\em J. Anal. Math.}, 2026.
\newblock To appear.

\bibitem{kulik:soulier:2020}
R.~Kulik and P.~Soulier, {\em {Heavy tailed time series.}}
\newblock Cham: Springer, 2020.

\bibitem{K2010}
Z.~Kabluchko, ``Stationary systems of {G}aussian processes,'' {\em Ann. Appl. Probab.}, vol.~20, no.~6, pp.~2295--2317, 2010.

\bibitem{MolchanovSPA}
I.~Molchanov and K.~Stucki, ``Stationarity of multivariate particle systems,'' {\em Stochastic Process. Appl.}, vol.~123, no.~6, pp.~2272--2285, 2013.

\bibitem{debicki2017approximation}
K.~D\c{e}bicki and E.~Hashorva, ``Approximation of supremum of max-stable stationary processes \& {P}ickands constants,'' {\em J. Theoretical Probability}, vol.~33, no.~1, pp.~444--464, 2020.

\bibitem{hashorva2025cluster}
E.~Hashorva, ``Cluster random fields and random-shift representations,'' {\em Journal of Theoretical Probability}, vol.~38, no.~3, p.~50, 2025.

\bibitem{KumeE}
E.~Hashorva and A.~Kume, ``Multivariate max-stable processes and homogeneous functionals,'' {\em Statist. Probab. Lett.}, vol.~173, p.~109066, 2021.

\bibitem{Wijsman1990}
R.~A. Wijsman, {\em Invariant Measures on Groups and Their Use in Statistics}, vol.~14 of {\em IMS Lecture Notes--Monograph Series}.
\newblock Hayward, CA: Institute of Mathematical Statistics, 1990.

\bibitem{KamiyaTakemuraKuriki2008}
H.~Kamiya, A.~Takemura, and S.~Kuriki, ``Star-shaped distributions and their generalizations,'' {\em Journal of Statistical Planning and Inference}, vol.~138, no.~11, pp.~3429--3447, 2008.

\bibitem{BojanPhilippe}
B.~Basrak, H.~Planinic, and P.~Soulier, ``An invariance principle for sums and record times of regularly varying stationary sequences,'' {\em Probab. Theory Relat. Fields}, vol.~172, p.~869–914, 2018.

\bibitem{Strokorb}
K.~Strokorb, F.~Ballani, and M.~Schlather, ``Tail correlation functions of max-stable processes: construction principles, recovery and diversity of some mixing max-stable processes with identical {TCF},'' {\em Extremes}, vol.~18, no.~2, pp.~241--271, 2015.

\bibitem{deHaan}
L.~de~Haan, ``A spectral representation for max-stable processes,'' {\em Ann. Probab.}, vol.~12, no.~4, pp.~1194--1204, 1984.

\bibitem{wang2025family}
Y.~Wang, ``A family of log-correlated gaussian processes,'' {\em Journal of Theoretical Probability}, vol.~38, no.~4, p.~82, 2025.

\bibitem{Kallenberg2014StationaryDensities}
O.~Kallenberg, ``Stationary and invariant densities and disintegration kernels,'' {\em Probab. Theory Related Fields}, vol.~160, pp.~567--592, 2014.

\bibitem{AranoIsonoMarrakchi2021}
Y.~Arano, Y.~Isono, and A.~Marrakchi, ``Ergodic theory of affine isometric actions on {Hilbert} spaces,'' {\em Geom. Funct. Anal.}, vol.~31, no.~5, pp.~1013--1094, 2021.

\bibitem{Kallenberg2007Invariant}
O.~Kallenberg, ``Invariant measures and disintegrations with applications to {Palm} and related kernels,'' {\em Probab. Theory Related Fields}, vol.~139, pp.~285--310, 2007.

\bibitem{Kallenberg2007InvariantErratum}
O.~Kallenberg, ``Invariant measures and disintegrations with applications to {Palm} and related kernels,'' {\em Probab. Theory Related Fields}, vol.~139, p.~311, 2007.
\newblock Erratum.

\end{thebibliography}
\end{document}